\documentclass[12pt,reqno]{amsart}
\usepackage{fullpage,url,amssymb,enumerate}
\usepackage[all]{xy} % for complicated commutative diagrams
\usepackage{mathrsfs} % for \mathscr (script letters)
\usepackage{bm} % for bold Greek letters

\usepackage[OT2,T1]{fontenc}
\DeclareSymbolFont{cyrletters}{OT2}{wncyr}{m}{n}
\DeclareMathSymbol{\Sha}{\mathalpha}{cyrletters}{"58}

\usepackage{color}

\newcommand{\Adual}{\widehat{A}}

\DeclareMathOperator{\NS}{NS}

\newcommand{\sheafU}{\mathscr{U}}

\newcommand{\AbC}{\mathfrak{Ab}_{\calC}}
\newcommand{\GpC}{\mathfrak{Gp}_{\calC}}
\bmdefine\boldmu{\mu}

\newcommand{\Bigwedge}[1]{\sideset{}{^{#1}}\bigwedge}

\newcommand{\defi}[1]{\textsf{#1}} % for defined terms

\newcommand{\C}{{\mathbb C}}
\newcommand{\F}{{\mathbb F}}
\newcommand{\G}{{\mathbb G}}

\newcommand{\Q}{{\mathbb Q}}
\newcommand{\R}{{\mathbb R}}
\newcommand{\Z}{{\mathbb Z}}

\newcommand{\kbar}{{\overline{k}}}

\newcommand{\qq}{{\mathfrak q}}

\newcommand{\calC}{{\mathcal C}}

\newcommand{\calT}{{\mathcal T}}

\newcommand{\calW}{{\mathcal W}}

\newcommand{\FF}{{\mathscr F}}

\newcommand{\LL}{{\mathscr L}}
\newcommand{\OO}{{\mathscr O}}

\DeclareMathOperator{\Char}{char}

\DeclareMathOperator{\Hom}{Hom}
\DeclareMathOperator{\HOM}{\bf Hom}
\DeclareMathOperator{\Ext}{Ext}
\DeclareMathOperator{\EXT}{\bf Ext}
\DeclareMathOperator{\Aut}{Aut}
\DeclareMathOperator{\Gal}{Gal}

\DeclareMathOperator{\Sym}{Sym}

\DeclareMathOperator{\Pic}{Pic}
\DeclareMathOperator{\Jac}{Jac}

\DeclareMathOperator{\PIC}{\bf Pic}
\DeclareMathOperator{\Spec}{Spec}

\newcommand{\ab}{{\operatorname{ab}}}

\newcommand{\tH}{{\operatorname{th}}}

\newcommand{\Cech}{\v{C}ech}
\newcommand{\HH}{{\operatorname{H}}}
\newcommand{\HHcech}{{\check{\HH}}}

\newcommand{\surjects}{\twoheadrightarrow}

\newcommand{\isom}{\simeq}

\newcommand{\tensor}{\otimes} % binary tensor product
\newcommand{\directsum}{\oplus} % binary direct sum
\newcommand{\Directsum}{\bigoplus} % direct sum of a collection

\newtheorem{theorem}{Theorem}[section]
\newtheorem{lemma}[theorem]{Lemma}
\newtheorem{corollary}[theorem]{Corollary}
\newtheorem{proposition}[theorem]{Proposition}

\theoremstyle{definition}
\newtheorem{definition}[theorem]{Definition}

\newtheorem{example}[theorem]{Example}

\theoremstyle{remark}

\newtheorem{remark}[theorem]{Remark}

\usepackage[
	backref,
	pdfauthor={Bjorn Poonen}, % add other authors
]{hyperref}
\usepackage[alphabetic,backrefs,lite]{amsrefs} % for bibliography

\begin{document}

\title{Self cup products and the theta characteristic torsor}
\subjclass[2010]{Primary 18G50; Secondary 11G10, 11G30, 14G25, 14K15}
\keywords{Weil pairing, theta characteristic, self cup product}
\author{Bjorn Poonen}
\thanks{B.~P.\ was partially supported by NSF grant DMS-0841321.}
\address{Department of Mathematics, Massachusetts Institute of Technology, Cambridge, MA 02139-4307, USA}
\email{poonen@math.mit.edu}
\urladdr{http://math.mit.edu/~poonen/}
\author{Eric Rains}
\address{Department of Mathematics, California Institute of Technology, Pasadena, CA 91125}
\email{rains@caltech.edu}
\date{April 10, 2011}

\begin{abstract}
We give a general formula relating self cup products
in cohomology to connecting maps in nonabelian cohomology, 
and apply it to obtain a formula for the self cup
product associated to the Weil pairing.
\end{abstract}

\maketitle

%****************************************************************************
\section{Introduction}
\label{S:introduction}

We prove a general statement about cohomology, 
Theorem~\ref{T:cohomology pairing},
that reinterprets the self cup product map 
\begin{align*}
	H^1(M) &\to H^2(M \tensor M) \\
	x &\mapsto x \cup x 
\end{align*}
as the connecting map of cohomology for a certain sequence
\[
	1 \to M \tensor M \to \sheafU M \to M \to 1
\]
involving a canonical nonabelian central extension $\sheafU M$.
The proof of Theorem~\ref{T:cohomology pairing} 
can be read with group cohomology in mind,
but we prove it for an arbitrary site
since later we need it for fppf cohomology.
The sheaf $\sheafU M$ has other properties as well:
for example, there is a map $M \to (\sheafU M)^{\ab}$
that is universal for quadratic maps from $M$ to a (variable) abelian sheaf.

As an application of Theorem~\ref{T:cohomology pairing},
we answer a question of B.~Gross
about the self cup product associated to the Weil pairing 
on the $2$-torsion of the Jacobian $A:=\Jac X$ of a curve $X$.
The Weil pairing
\[
	e_2 \colon A[2] \times A[2] \to \G_m
\]
induces a symmetric bilinear pairing
\[
	\langle\;,\;\rangle \colon 
	\HH^1(A[2]) \times \HH^1(A[2]) \to \HH^2(\G_m).
\]
We prove the identity 
\begin{equation}
\label{E:x,x=x,c}  
	\langle x, x \rangle = \langle x,c_\calT \rangle,
\end{equation}
where $c_\calT$ is a particular canonical element of $\HH^1(A[2])$.
Namely, $c_\calT$ is the class of the torsor under $A[2]$ parametrizing the
theta characteristics on $X$.

We have been vague about the field of definition of our curve;
in fact, it is not too much harder to work over an arbitrary base scheme.
(See Theorem~\ref{T:cohomology pairing on Jacobian}.)
Moreover, we prove a version with Jacobians replaced by arbitrary
abelian schemes $A$,
in which $c_\calT$ is replaced by an element $c_\lambda \in \HH^1(\Adual[2])$.
(See Theorem~\ref{T:cohomology pairing on abelian scheme}.)

With an eye towards applications of these theorems,
we give many criteria for the vanishing of $c_\lambda$ and $c_\calT$,
some of which generalize earlier results of M.~Atiyah and D.~Mumford:
see Proposition~\ref{P:c_lambda=0} and Remark~\ref{R:Atiyah and Mumford}.
We also give an example over $\Q_3$ for which $c_\calT \ne 0$,
and an example over $\Q$ for which $c_\calT$ is nonzero but locally trivial.

As further motivation, some of our results,
namely Proposition~\ref{P:commutator pairing} 
and Theorem~\ref{T:cohomology pairing on Jacobian},
are used in~\cite{Poonen-Rains-selmer-preprint} 
to study the distribution of Selmer groups.

%****************************************************************************
\section{Some homological algebra}
\label{S:quadratic maps}

\subsection{A tensor algebra construction}
\label{S:construction}

We will define a functor $U$ from the category of $\Z$-modules
to the category of groups,
the goal being Theorem~\ref{T:cohomology pairing}.
Let $M$ be a $\Z$-module.
Let $TM = \Directsum_{i \ge 0} T^iM$ be the tensor algebra.
Then $T^{\ge n}M = \Directsum_{i \ge n} T^iM$ is a 2-sided ideal of $TM$.
Let $T^{<n}M$ be the quotient ring $TM/T^{\ge n}M$.
Let $UM$ be the kernel of 
$(T^{<3}M)^\times \to (T^{<1}M)^{\times} = \Z^\times = \{\pm 1\}$.
The grading on $TM$ gives rise to a filtration of $UM$,
which yields the following central extension of groups
\begin{equation}
\label{E:Q group}
	1 \to M \tensor M \to UM \stackrel{\pi}\to M \to 1.
\end{equation}
Elements of $UM$ may be written as $1+m+t$ 
where $m \in M$ and $t \in M \tensor M$,
and should be multiplied as follows:
\[
	(1+m+t)(1+m'+t') = 1 + (m+m') + ((m \tensor m')+t+t').
\]
The surjection $UM \to M$ admits 
a set-theoretic section $s\colon M \to UM$ sending $m$ to $1+m$.
If $m,m' \in M$, then 
\begin{equation}
\label{E:cocycle of section}
	s(m) \; s(m') \; s(m+m')^{-1} = m \tensor m' 
\end{equation}
in $M \tensor M \subseteq UM$.

A simple computation verifies the following universal property of $UM$:
\begin{proposition}
\label{P:universal property}
The map $s \colon M \to UM$ is universal for 
set maps $\sigma \colon M \to G$ to a group $G$
such that $(m,m')\mapsto \sigma(m) \sigma(m')\sigma(m+m')^{-1}$
is a bilinear function from $M \times M$ to an abelian subgroup of $G$.
\end{proposition}

A \defi{quadratic map} $q \colon M \to G$ 
is a set map between abelian groups
such that $(m,m') \mapsto q(m+m')-q(m)-q(m')$ is bilinear.
(Perhaps ``{\em pointed} quadratic map'' would be
better terminology; 
for instance, the quadratic maps $q \colon \Q \to \Q$ 
are the polynomial functions of degree at most $2$ sending $0$ to $0$.)
Proposition~\ref{P:universal property} implies:

\begin{corollary}
\label{C:universal property for quadratic maps}
The map $M \to (UM)^\ab$ is universal for quadratic maps
from $M$ to an abelian group.
The map $M \to (UM)^\ab \tensor \F_2$ is universal for quadratic maps
from $M$ to an abelian group such that the image is killed by $2$.
\end{corollary}

\begin{remark}\label{R:abelianization of UM}
\hfill
\begin{enumerate}[\upshape (a)]
\item 
The commutator $[1+m+t,1+m'+t]$ equals $m \tensor m' - m' \tensor m$,
so we have an exact sequence of abelian groups
\[
	0 \to S^2 M \to (UM)^\ab \to M \to 0,
\]
where $S M = \Directsum_{n \ge 0} S^n M$ is the symmetric algebra.
In particular, 
\[
	(UM)^\ab \isom 
	\ker\left( (S^{<3}M)^\times \to (S^{<1}M)^\times \right).
\]
\item
Similarly, if $2M=0$, then $(1+m+t)^2=1+m \tensor m$, 
so we obtain an exact sequence of $\F_2$-vector spaces
\[
	0 \to \Bigwedge{2} M \to (UM)^\ab \tensor \F_2 \to M \to 0,
\]
and
\[
	(UM)^\ab \tensor \F_2 \isom 
	\ker\left( (\Bigwedge{<3}M)^\times \to (\Bigwedge{<1}M)^\times \right).
\]
\end{enumerate}
\end{remark}

\subsection{Sheaves of groups}
\label{S:sheaves of groups}

In the rest of Section~\ref{S:quadratic maps}, $\calC$ is a site.
Let $\GpC$ be the category of sheaves of groups on $\calC$,
and let $\AbC$ be the category of sheaves of abelian groups on $\calC$.
For $M \in \AbC$, write $\HH^i(M)$ for $\Ext^i(\Z,M)$,
where $\Z$ is the constant sheaf;
in other words, $\HH^i(-)$ is the $i^\tH$ right derived functor
of $\Hom(\Z,-)$ on $\AbC$.
For $M \in \GpC$, define $\HH^0(M)$ as $\Hom(\Z,M)$
and define $\HH^1(M)$ in terms of torsors 
as in \cite{Giraud1971}*{\S III.2.4}.
The definitions are compatible 
for $M \in \AbC$ and $i=0,1$ \cite{Giraud1971}*{Remarque~III.3.5.4}.

\begin{remark}
The reader may prefer to imagine the case for which
sheaves are $G$-sets for some group $G$,
abelian sheaves are $\Z G$-modules,
and $\HH^i(M)$ is just group cohomology.
\end{remark}

All the constructions and results of Section~\ref{S:construction}
have sheaf analogues.
In particular, 
for $M \in \AbC$ we obtain $\sheafU M \in \GpC$ fitting in exact sequences
\begin{gather}
\label{E:U sheaf}
	1 \to M \tensor M \to \sheafU M \to M \to 1 \\
\label{E:U^ab sheaf}
	0 \to S^2 M \to (\sheafU M)^\ab \to M \to 0,
\end{gather}
and, if $2M=0$,
\begin{equation}
\label{E:U^ab tensor F2}
	0 \to \Bigwedge{2} M \to (\sheafU M)^\ab \tensor \F_2 \to M \to 0.
\end{equation}

\subsection{Self cup products}
\label{S:cup products}

\begin{theorem}
\label{T:cohomology pairing}
For $M \in \AbC$, the connecting map $\HH^1(M) \to \HH^2(M \tensor M)$
induced by \eqref{E:U sheaf} (see \cite{Giraud1971}*{\S IV.3.4.1})
maps each $x$ to $x \cup x$.
\end{theorem}

\begin{proof}
Let $x \in \HH^1(M)=\Ext^1(\Z,M)$.
Let
\begin{equation}
\label{E:X extension}
	0 \to M \to X \stackrel{\alpha}\to \Z \to 0
\end{equation}
be the corresponding extension.
Let $X_1:=\alpha^{-1}(1)$, which is a sheaf of torsors under $M$.

We will construct a commutative diagram 
\begin{equation}
\label{E:extension diagram}
\xymatrix{
1 \ar[r] & M \tensor M \ar[r] & \sheafU M \ar[r]^-{\pi} & M \ar[r] & 1 \\
1 \ar[r] & M \tensor M \ar[r] \ar@{=}[u] \ar@{=}[d] & G \ar[r]^-{\delta} \ar[u] \ar[d] & G' \ar[r] \ar[u] \ar[d] & 1 \\
0 \ar[r] & M \tensor M \ar[r] & X \tensor M \ar[r]^-{\epsilon} & M \ar[r] & 0 \\
}
\end{equation}
of sheaves of groups, with exact rows.
The first row is \eqref{E:U sheaf}.
The last row, obtained by tensoring \eqref{E:X extension} with $M$, 
is exact since $\Z$ is flat.
Let $G$ be the sheaf of $(u,t) \in \sheafU M \directsum (X \tensor M)$ 
such that $\pi(u)=\epsilon(t)$ in $M$.
The vertical homomorphisms emanating from $G$ are the two projections.
Let $\delta \colon G \to \sheafU X$ send $(u,t)$ to $u-t$.
Then $\ker \delta = M \tensor M$, embedded diagonally in $G$.
Let $G'=\delta(G)$.
Explicitly, if $e$ is a section of $X_1$, then $G'$ consists of 
sections of $\sheafU X$ of the form $1+m-e \tensor m + t$ with $m \in M$ and $t \in M \tensor M$.
The vertical homomorphisms emanating from $G'$ are induced by
the map $G \to M$ sending $(u,t)$ to $\pi(u)=\epsilon(t)$.

A calculation shows that $1+X_1 + M \tensor M$ is a right torsor $X'$ under $G'$,
corresponding to some $x' \in \HH^1(G')$.
Moreover, $\sheafU X \to X$ restricts to a {\em torsor} map $X' \to X_1$ 
compatible with $G' \to M$,
so $\HH^1(G') \to \HH^1(M)$ sends $x'$ to $x$.

By \cite{Giraud1971}*{\S IV.3.4.1.1}, 
the commutativity of \eqref{E:extension diagram} shows that 
the image of $x$ under the connecting map
$\HH^1(M) \to \HH^2(M \tensor M)$ from the first row,
equals the image of $x'$ under the connecting map
$\HH^1(G') \to \HH^2(M \tensor M)$ from the second row,
which equals the image of $x$ under the connecting homomorphism
$\HH^1(M) \to \HH^2(M \tensor M)$ from the third row.
This last homomorphism is $y \mapsto x \cup y$,  
so it maps $x$ to $x \cup x$ 
(cf.~\cite{Yoneda1958}, which explains 
this definition of $x \cup y$ 
for extensions of modules over a ring).
\end{proof}

\begin{example}
If $M=\Z/2\Z$, then \eqref{E:U sheaf} is the sequence of constant sheaves
\[
	0 \to \Z/2\Z \to \Z/4\Z \to \Z/2\Z \to 0,
\]
which induces the \defi{Bockstein morphism} 
$\HH^1(X,\Z/2\Z) \to \HH^2(X,\Z/2\Z)$.
So Theorem~\ref{T:cohomology pairing} recovers the known result
that for any topological space $X$, 
the self-cup-product $\HH^1(X,\Z/2\Z) \to \HH^2(X,\Z/2\Z)$ 
is the Bockstein morphism.
(See properties (5) and (7) of Steenrod squares in Section 4.L 
of \cite{Hatcher-Algebraic-Topology}.)
\end{example}

\begin{remark}
\label{R:Cech cochains}
In group cohomology, if we represent a class in $\HH^1(M)$ by a cochain  
$\zeta$, then one can check that the coboundary of $s\circ\zeta$ 
equals the difference of $\zeta\cup \zeta$ and the image of $\zeta$
under the connecting map.  A similar argument using \Cech\ cochains
gives an alternate proof of the general case of 
Theorem~\ref{T:cohomology pairing}, as we now explain.

By \cite{Giraud1971}*{Th\'eor\`eme~0.2.6},
we may replace $\calC$ by a site with one having an equivalent topos
in order to assume that $\calC$ has finite fiber products,
and in particular, a final object $S$.
Then the natural map $\HHcech^1(M) \to \HH^1(M)$ is an isomorphism
\cite{Giraud1971}*{Remarque~III.3.6.5(5)},
so any $x \in \HH^1(M)$ is represented 
by a $1$-cocycle $m$ for some covering $(S'_i \to S)_{i \in I}$.
For simplicity, let us assume that the covering
consists of {\em one} morphism $S' \to S$ 
(the general case is similar).
Let $S'' = S' \times_S S'$, and $S''' = S' \times_S S' \times_S S'$.
Let $\pi_{23},\pi_{13},\pi_{12} \colon S''' \to S''$ be the projections.
So $m \in M(S'')$ satisfies $\pi_{13}^*m = \pi_{12}^*m + \pi_{23}^*m$.
Applying the section $M \to \sheafU M$ yields
a $1$-cochain $1+m \in (\sheafU M)(S'')$.
Its $2$-coboundary in $(M \tensor M)(S''') \subseteq (\sheafU M)(S''')$ 
represents the image of $x$ under the connecting map
$\HH^1(M) \to \HH^2(M \tensor M)$ \cite{Giraud1971}*{Corollaire~IV.3.5.4(ii)}.
On the other hand, the definition of the $2$-coboundary given
in \cite{Giraud1971}*{Corollaire~IV.3.5.4}
together with \eqref{E:cocycle of section}
shows that it is $\pi^*_{12}m \tensor \pi^*_{23}m \in (M \tensor M)(S''')$,
whose class in $\HH^2(M \tensor M)$ represents $x \cup x$, by definition.
\end{remark}

Let $M,N \in \AbC$.
Let $\beta \in \HH^0(\HOM(M \tensor M,N))$.
(The bold face in $\HOM$, $\EXT$, etc., indicates that we mean
the sheaf versions.)
Using $\beta$, construct an exact sequence
\begin{equation}
  \label{E:sheafU_beta}
	0 \to N \to \sheafU_\beta \to M \to 0
\end{equation}
as the pushout of~\eqref{E:U sheaf} by $\beta \colon M \tensor M \to N$.

If $\beta$ is symmetric, then $\sheafU_\beta$ is abelian,
and we let $\epsilon_\beta$ be the class of \eqref{E:sheafU_beta}
in $\Ext^1(M,N)$.
If $\beta$ is symmetric and $\EXT^1(M,N)=0$, 
then applying $\HOM(-,N)$ to \eqref{E:U^ab sheaf} yields
\[
	0 \to \HOM(M,N) \to \HOM((\sheafU M)^\ab,N) \to \HOM(S^2 M,N) \to 0
\]
and a connecting homomorphism sends $\beta \in \HH^0(\HOM(S^2 M,N))$
to an element $c_\beta \in \HH^1(\HOM(M,N))$.

\begin{corollary}
\label{C:4 homomorphisms}
Then the following maps $\HH^1(M) \to \HH^2(N)$ are the same,
when defined:
\begin{enumerate}[\upshape (a)]
\item
\label{I:4 homomorphisms cup product}
The composition
\[
	\HH^1(M) \stackrel{\Delta}\longrightarrow
	\HH^1(M) \times \HH^1(M) \stackrel{\cup}\longrightarrow
	\HH^2(M \tensor M) \stackrel{\beta}\longrightarrow
	\HH^2(N).
\]
\item 
\label{I:4 homomorphisms connecting}
The connecting homomorphism $\HH^1(M) \to \HH^2(N)$ associated 
to~\eqref{E:sheafU_beta}.
\item
\label{I:4 homomorphisms Yoneda}
The pairing with $\epsilon_\beta$ under the Yoneda product
\[
	\Ext^1(M,N) \times \HH^1(M) \to \HH^2(N)
\]
(if $\beta$ is symmetric).
\item
\label{I:4 homomorphisms evaluation}
The pairing with $c_\beta$ under the evaluation cup product
\[
	\HH^1(\HOM(M,N)) \times \HH^1(M) \to \HH^2(N)
\]
(if $\beta$ is symmetric and $\EXT^1(M,N)=0$).
\end{enumerate}
\end{corollary}

\begin{proof}
Theorem~\ref{T:cohomology pairing}
and functoriality implies
the equality of \eqref{I:4 homomorphisms cup product}
and \eqref{I:4 homomorphisms connecting}.
Standard homological algebra gives equality of
\eqref{I:4 homomorphisms connecting},
\eqref{I:4 homomorphisms Yoneda}, and
\eqref{I:4 homomorphisms evaluation}.
\end{proof}

\subsection{Commutator pairings}
\label{S:commutator pairings}

\begin{proposition}
\label{P:commutator pairing}
Let $1 \to A \to B \stackrel{\rho}\to C \to 1$ be an exact sequence
in $\GpC$, with $A$ central in $B$, and $C$ abelian.
Let $q \colon \HH^1(C) \to \HH^2(A)$ be the connecting map.
Given $c_1,c_2 \in C$, we can lift them locally to $b_1,b_2$
and form their commutator $[b_1,b_2]:=b_1 b_2 b_1^{-1} b_2^{-1} \in A$;
this induces a homomorphism $[\;,\;]\colon C \tensor C \to A$.
For $\gamma_1,\gamma_2 \in \HH^1(C)$,
we have that $q(\gamma_1+\gamma_2)-q(\gamma_1)-q(\gamma_2)$ 
equals the image of $-\gamma_1 \cup \gamma_2$ 
under the homomorphism $\HH^2(C \tensor C) \to \HH^2(A)$
induced by $[\;,\;]$.
\end{proposition}

\begin{proof}
That the commutator induces a homomorphism is a well-known
simple computation.
Pulling back $1 \to A^3 \to B^3 \to C^3 \to 1$
by the homomorphism $C^2 \to C^3$ sending $(c_1,c_2)$ to $(c_1,c_2,c_1+c_2)$
and then pushing out 
by the homomorphism $A^3 \to A$ sending $(a_1,a_2,a_3)$ to $a_3-a_2-a_1$
yields an exact sequence $1 \to A \to Q \to C^2 \to 1$.
Here $Q = B'/B''$ 
where $B'$ is the subgroup sheaf of $(b_1,b_2,b_3) \in B^3$
satisfying $\rho(b_3)=\rho(b_1)+\rho(b_2)$,
and $B''$ is the subgroup sheaf of $B^3$ generated by 
sections $(a_1,a_2,a_3) \in A^3$ with $a_3=a_1+a_2$.
The surjection $Q \to C^2$ admits a section $\sigma \colon C^2 \to Q$
defined locally as follows: given $(c_1,c_2)$ lifting to 
$(b_1,b_2) \in B^2$,
send it to the image of $(b_1,b_2,b_1 b_2)$ in $Q$
(this is independent of the choice of lifts, since we work modulo $B''$).
A calculation shows that
\begin{equation}
\label{E:cocycle}
\sigma((c_1',c_2')) \; \sigma((c_1,c_2) + (c_1',c_2'))^{-1} \; \sigma((c_1,c_2))
= [c_1',c_2^{-1}] = -[c_1',c_2]
\end{equation}
in $A$, and the three factors on the left may be rotated
since the right hand side is central in $Q$.
Proposition~\ref{P:universal property} and \eqref{E:cocycle}
yield the middle vertical map
in the commutative diagram 
\begin{equation}
\label{E:weird}
\xymatrix{
1 \ar[r] & C^2 \tensor C^2 \ar[r] \ar[d] & \sheafU(C^2) \ar[r] \ar[d] & C^2 \ar[r] \ar@{=}[d] & 1 \\
1 \ar[r] & A \ar[r] & Q \ar[r] & C^2 \ar[r] & 1 \\
}  
\end{equation}
with exact rows,
and the left vertical map sends $(c_1,c_2) \tensor (c_1',c_2')$ 
to $-[c_1',c_2]$.
The connecting map for the first row sends
$(\gamma_1,\gamma_2) \in \HH^1(C^2)$
to $(\gamma_1,\gamma_2) \cup (\gamma_1,\gamma_2) \in \HH^2(C^2 \tensor C^2)$,
by Theorem~\ref{T:cohomology pairing}.
The connecting map for the second row
is a composition $\HH^1(C^2) \to \HH^1(C^3) \to \HH^2(A^3) \to \HH^2(A)$,
so it maps $(\gamma_1,\gamma_2)$
to $q(\gamma_1+\gamma_2)-q(\gamma_1)-q(\gamma_2)$.
Finally, the left vertical map
sends $(\gamma_1,\gamma_2) \cup (\gamma_1,\gamma_2) \in \HH^2(C^2 \tensor C^2)$
to the image of $-\gamma_1 \cup \gamma_2$ under the commutator pairing
$\HH^2(C \tensor C) \to \HH^2(A)$.
So compatibility of the connecting maps yields the result.
\end{proof}

\begin{remark}
Yu.~Zarhin~\cite{Zarhin1974} proved Proposition~\ref{P:commutator pairing} 
in the special case of group cohomology,
by an explicit calculation with cocycles.
Using the approach of Remark~\ref{R:Cech cochains},
that argument can be adapted to give a second proof of
Proposition~\ref{P:commutator pairing} in the general case.
\end{remark}

%****************************************************************************
\section{Abelian schemes}
\label{S:abelian schemes}

\subsection{The relative Picard functor}

Let $A \to S$ be an abelian scheme.
Let $\PIC_{A/S}$ be its relative Picard functor on the big fppf site of $S$.
Trivialization along the identity section shows that
$\PIC_{A/S}(T) \isom \Pic(A \times_S T)/\Pic T$ for each $S$-scheme $T$
(see Proposition~4 on page~204 of 
\cite{Bosch-Lutkebohmert-Raynaud1990}).
We generally identify line sheaves with their classes in $\PIC$.
For an $S$-scheme $T$ and $a \in A(T)$, let
\begin{align*}
	\tau_a\colon A_T &\to A_T \\
	x &\mapsto a+x
\end{align*}
be the translation-by-$a$ morphism.
Given a line sheaf $\LL$ on $A$, the theorem of the square implies that
\begin{align}
\label{E:definition of phi_L}
	\phi_\LL \colon A &\to \PIC_{A/S} \\
\nonumber
	a &\mapsto \tau_a^* \LL \tensor \LL^{-1}
\end{align}
is a homomorphism.
If we vary the base and vary $\LL$, we obtain a homomorphism of fppf-sheaves
\begin{align*}
	\PIC_{A/S} &\to \HOM(A,\PIC_{A/S}) \\
	\LL &\mapsto \phi_\LL.
\end{align*}
Its kernel is denoted $\PIC^0_{A/S}$.
Using the fact that $\PIC_{A/S}$ is an algebraic space,
and the fact that $\PIC^0_{A/S}$ is an open subfunctor of $\PIC_{A/S}$ 
(which follows from \cite{SGA6}*{Expos\'e~XIII,~Th\'eor\`eme~4.7}),
one can show that $\PIC^0_{A/S}$ is another abelian scheme $\Adual$ over $S$ 
\cite{Faltings-Chai1990}*{p.~3}.
The image of $\phi_\LL$ is contained in $\Adual$,
so we may view $\phi_\LL$ as a homomorphism $A \to \Adual$.
Moreover, $\phi_\LL$ equals its dual homomorphism $\widehat{\phi}_\LL$.
In fact, we have an exact sequence of fppf-sheaves
\begin{equation}
\label{E:NS sequence}
	0 \to \Adual \to \PIC_{A/S} \to \HOM_{\textup{self-dual}}(A,\Adual) \to 0.
\end{equation}

\begin{remark}
Let $k$ be a field, let $k_s$ be a separable closure
contained in an algebraic closure $\kbar$,
and let $G_k=\Gal(k_s/k)$.
For an abelian variety $A$ over $k$,
the group $\Hom_{\textup{self-dual}}(A,\Adual)$
of global sections of $\HOM_{\textup{self-dual}}(A,\Adual)$
may be identified with the $G_k$-invariant subgroup
of the N\'eron-Severi group $\NS A_{k_s}$.
(For the case $k=\kbar$ see \cite{MumfordAV1970}, in particular 
Corollary~2 on page~178
and Theorem~2 on page~188 and the remark following it.
The general case follows because any homomorphism
defined over $\kbar$ is in fact defined over $k_s$,
since it maps the Zariski-dense set of prime-to-$(\Char k)$ 
torsion points in $A(k_s)$ to points in $\Adual(k_s)$.)
\end{remark}

For any homomorphism of abelian schemes $\lambda \colon A \to B$,
let $A[\lambda]:=\ker \lambda$.

\subsection{Symmetric line sheaves}

Multiplication by an integer $n$ on $A$ induces a pullback homomorphism
$[n]^* \colon \PIC_{A/S} \to \PIC_{A/S}$.
Let $\PIC_{A/S}^{\Sym}$ be the kernel of $[-1]^*-[1]^*$ on $\PIC_{A/S}$.
More concretely, because $A \to S$ has a section, we have
$\PIC_{A/S}^{\Sym}(T) = \Pic^{\Sym}(A \times_S T)/\Pic T$
for each $S$-scheme $T$,
where $\Pic^{\Sym}(A \times_S T)$ is the group of
isomorphism classes of symmetric line sheaves on $A \times_S T$.
Since $[-1]^*$ acts as $-1$ on $\Adual$ and as $+1$ on
$\HOM_{\textup{self-dual}}(A,\Adual)$,
and since multiplication-by-$2$ on $\Adual$ is surjective,
the snake lemma applied to $[-1]^*-[1]^*$ acting on \eqref{E:NS sequence} 
yields an exact sequence
\begin{equation}
\label{E:Pic^Sym}
	0 \to \Adual[2] \to \PIC_{A/S}^{\Sym} 
		\to \HOM_{\textup{self-dual}}(A,\Adual) \to 0.
\end{equation}

\subsection{The Weil pairing}
\label{S:Weil pairing}

We recall some facts and definitions
that can be found in~\cite{Polishchuk2003}*{\S10.4}, for example.
(In that book, $S$ is $\Spec k$ for a field $k$, but the same arguments apply
over an arbitrary base scheme.)
Given an abelian scheme $A$ over $S$,
there is a \defi{Weil pairing}
\begin{equation}
\label{E:Weil pairing}
	e_2 \colon A[2] \times \Adual[2] \to \G_m.
\end{equation}
For any homomorphism $\lambda \colon A \to \Adual$,
define $e_2^\lambda \colon A[2] \times A[2] \to \G_m$
by $e_2^\lambda(x,y)=e_2(x,\lambda y)$.
If $\LL \in \PIC_{A/S}(S)$, let $e_2^\LL=e_2^{\phi_\LL}$;
this is an alternating bilinear pairing,
and hence it is also symmetric.

\subsection{Quadratic refinements of the Weil pairing}

\begin{proposition}
\label{P:properties of q}
There is a (not necessarily bilinear) pairing of fppf sheaves
\[
	\qq \colon A[2] \times \PIC_{A/S}^{\Sym} \to \boldmu_2 \subset \G_m
\]
such that:
\begin{enumerate}[\upshape (a)]
\item
\label{I:properties of q additivity} 
The pairing is additive in the second argument: 
$\qq(x,\LL \tensor \LL') = \qq(x,\LL) \qq(x,\LL')$.
\item
\label{I:properties of q Weil pairing} 
The restriction of $\qq$ to a pairing
\[
	\qq \colon A[2] \times \Adual[2] \to \G_m
\]
is the Weil pairing $e_2$.  In particular, this restriction is bilinear.
\item 
\label{I:properties of q quadratic} 
In general, $\qq(x+y,\LL)=\qq(x,\LL) \qq(y,\LL) e_2^\LL(x,y)$.
\end{enumerate}
\end{proposition}

\begin{proof}
See \cite{Polishchuk2003}*{\S13.1}.
\end{proof}

We can summarize Proposition~\ref{P:properties of q}
in the commutative diagram
\begin{equation}
\label{E:commutative diagram}
\xymatrix{
0 \ar[r] & \Adual[2] \ar[r] \ar[d]_{e_2}^{\wr} & \PIC_{A/S}^{\Sym} \ar[r] \ar[d]_{\qq} & \HOM_{\textup{self-dual}}(A,\Adual) \ar[r] \ar[d]_{e_2^\bullet} & 0 \\
0 \ar[r] & \HOM(A[2],\G_m) \ar[r] & \HOM((\sheafU A[2])^\ab \tensor \F_2,\G_m) \ar[r] & \textstyle{\HOM(\Bigwedge{2} A[2],\G_m)} \ar[r] & 0, \\
}
\end{equation}
which we now explain.
The top row is \eqref{E:Pic^Sym}.
The bottom row is obtained by applying $\HOM(-,\G_m)$
to \eqref{E:U^ab tensor F2} for $M=A[2]$,
and using $\EXT^1(A[2],\G_m)=0$ 
(a special case of~\cite{Waterhouse1971}*{Theorem~1, with the argument of \S3 to change fpqc to fppf}).
The vertical maps are 
the map $y \mapsto e_2(-,y)$,
the map sending $\LL$ to the homomorphism 
$(\sheafU A[2])^\ab \tensor \F_2 \to \G_m$
corresponding to $\qq(-,\LL)$ 
(see Corollary~\ref{C:universal property for quadratic maps}),
and the map $\lambda \mapsto e_2^\lambda$, respectively.
Commutativity of the two squares are given by 
\eqref{I:properties of q Weil pairing}
and
\eqref{I:properties of q quadratic}
in Proposition~\ref{P:properties of q}, respectively.

The top row of \eqref{E:commutative diagram} gives a homomorphism
\begin{align}
\nonumber  \Hom_{\textup{self-dual}}(A,\Adual) &\to \HH^1(\Adual[2]) \\
\label{E:c_lambda}	\lambda &\mapsto c_\lambda.
\end{align}
We may interpret $c_\lambda$ geometrically
as the class of the torsor under $\Adual[2]$ 
that parametrizes symmetric line sheaves $\LL$ with $\phi_\LL=\lambda$ 
(cf.~\cite{Polishchuk2003}*{\S13.5}).
Thus $c_\lambda$ is the obstruction to finding $\LL \in \Pic^{\Sym} A$ 
with $\phi_\LL=\lambda$.

\begin{remark}
\label{R:c_lambda and c_lambda}
The map $\HH^1(\Adual[2]) \to \HH^1(\Adual)$
sends $c_\lambda$ to the element called $c_\lambda$
in~\cite{Poonen-Stoll1999}*{\S4}.
If $k$ is a global field and $S=\Spec k$,
then this and~\cite{Poonen-Stoll1999}*{Corollary~2}
imply that our $c_\lambda$ lies in the $2$-Selmer group of $\Adual$.
\end{remark}

\begin{theorem}
\label{T:cohomology pairing on abelian scheme}
For any $\lambda \in \Hom_{\textup{self-dual}}(A,\Adual)$.
and any $x \in \HH^1(A[2])$, we have
\begin{equation}
\label{E:cup products for abelian scheme}
	x \underset{e_2^\lambda}{\cup} x = x \underset{e_2}{\cup} c_\lambda
\end{equation}
in $\HH^2(\G_m)$, where the cup products are induced
by the pairings underneath.
\end{theorem}

\begin{proof}
The rightmost vertical map in \eqref{E:commutative diagram} maps
$\lambda$ to $e_2^\lambda$.
These are mapped by the horizontal connecting homomorphisms
to $c_\lambda \in \HH^1(\Adual[2])$ and 
$c_{e_2^\lambda} \in \HH^1(\HOM(A[2],\G_m))$,
which are identified by the leftmost vertical map $e_2$.
Apply Corollary~\ref{C:4 homomorphisms} with
$M=A[2]$, $N=\G_m$, and $\beta=e_2^\lambda$,
using $\EXT^1(A[2],\G_m)=0$:
map~\eqref{I:4 homomorphisms cup product}
gives the left hand side of \eqref{E:cup products for abelian scheme}
and map~\eqref{I:4 homomorphisms evaluation}
gives the right hand side of \eqref{E:cup products for abelian scheme}
(written backwards) because of the identification of $c_\lambda$ 
with $c_{e^\lambda}$ via $e_2$.
\end{proof}

\subsection{Criteria for triviality of the obstruction}
\label{S:criteria}

The following lemma serves only to prove 
Proposition~\ref{P:c_lambda=0}\eqref{I:factors through cyclic} below.

\begin{lemma}
\label{L:finite cyclic group}
Let $k$ be a field, and let $G$ be a finite cyclic group.
\begin{enumerate}[\upshape (a)]
\item\label{I:same dimension as dual}
Let $A$ be a finite-dimensional $kG$-module.
Let $A^*:=\Hom_k(A,k)$ be the dual representation,
and let $A^G$ be the subspace of $G$-invariant vectors.
Then $\dim A^G = \dim (A^*)^G$.
\item\label{I:G-set section}
If $0 \to A \to B \to C \to 0$ is an exact sequence of
finite-dimensional $kG$-modules,
and the surjection $B^* \to A^*$ admits a section as $G$-sets,
then the connecting homomorphism $C^G \to \HH^1(G,A)$ is $0$.
\end{enumerate}
\end{lemma}

\begin{proof}\hfill
\begin{enumerate}[\upshape (a)]
\item Let $g$ be a generator of $G$.  If $M$ is a matrix representing
  the action of $g$ on $A$, the action of $g^{-1}$ on $A^*$ is given
  by the transpose $M^t$.  Then $\dim A^G = \dim \ker(M-1) = \dim
  \ker(M^t-1) = \dim (A^*)^G$, where the middle equality uses the fact
  that a matrix has the same rank as its transpose.
\item 
The section gives the $0$ at the right in
\[
	0 \to (C^*)^G \to (B^*)^G \to (A^*)^G \to 0.
\]
Taking dimensions and applying \eqref{I:same dimension as dual}
yields $\dim B^G = \dim A^G + \dim C^G$.
This together with the exactness of
\[
	0 \to A^G \to B^G \to C^G \to \HH^1(G,A)
\]
implies that the connecting homomorphism $C^G \to \HH^1(G,A)$ is $0$.\qedhere
\end{enumerate}
\end{proof}

\begin{proposition}
\label{P:c_lambda=0}
Let $\lambda \colon A \to \Adual$ be a self-dual homomorphism
of abelian varieties over a field $k$.
Suppose that at least one of the following hypotheses holds:
\begin{enumerate}[\upshape (a)]
\item\label{I:factors through cyclic} 
$\Char k \ne 2$ and the image $G$ of $G_k \to \Aut A[2](k_s)$ is cyclic.
\item\label{I:cT for perfect k of char 2}
$k$ is a perfect field of characteristic $2$.
\item\label{I:cT for archimedean k}
$k$ is $\R$ or $\C$.
\item\label{I:cT for local k}
$k$ is a nonarchimedean local field of residue characteristic not $2$,
and $A$ has \defi{good reduction}
(i.e., extends to an abelian scheme over the valuation ring of $k$).
\item\label{I:cT for finite k}
$k$ is a finite field.
\item\label{I:at most 4}
$\lambda(A[2])$ is an \'etale group scheme of rank at most $4$.
\end{enumerate}
Then $c_\lambda=0$.
\end{proposition}

\begin{proof}\hfill
\begin{enumerate}[\upshape (a)]
  \item 
Apply Lemma~\ref{L:finite cyclic group}\eqref{I:G-set section}
to the bottom row of \eqref{E:commutative diagram},
viewed as a sequence of $\F_2 G$-modules;
it applies since the dual sequence is~\eqref{E:U^ab tensor F2} for $M:=A[2]$,
and the section $s$ of Section~\ref{S:construction}
yields a $G$-set section $A[2] \to (\sheafU A[2])^\ab \tensor \F_2$.
Thus the top horizontal map in
\[
\xymatrix{
\HH^0(G,\HOM(\Bigwedge{2} A[2],\G_m)) \ar[r] \ar@{=}[d] 
	& \HH^1(G,\Adual[2]) \ar[d] \\
\HH^0(G_k,\HOM(\Bigwedge{2} A[2],\G_m)) \ar[r]^-\delta
	& \HH^1(G_k,\Adual[2]) \\
}
\]
is $0$.
Thus $\delta=0$.
Now \eqref{E:commutative diagram} shows that $c_\lambda=\delta(e_2^\lambda)=0$.
\item 
Let $M:=A[2]$, and let $M^\vee:=\HOM(M,\G_m)=\Adual[2]$ be its Cartier dual.
The bottom row of~\eqref{E:commutative diagram} yields an exact sequence
\begin{equation}
\label{E:H^0 and H^0}
	\HH^0(\HOM((\sheafU M)^\ab \tensor \F_2,\G_m)) 
	\to \HH^0(\textstyle{\HOM(\Bigwedge{2} M,\G_m)})
	\stackrel{\delta}\to \HH^1(M^\vee).
\end{equation}
It suffices to prove that $\delta=0$, or that the first map is surjective.
Equivalently, by the universal property of $(\sheafU M)^\ab \tensor \F_2$,
we need each alternating pairing $b \colon M \times M \to \boldmu_2$
to be $q(x+y)-q(x)-q(y)$ for some quadratic map $q \colon M \to \boldmu_2$.

In fact, we will prove this for {\em every} 
finite commutative group scheme $M$ over $k$ with $2M=0$.
Since $k$ is perfect, there is a canonical decomposition 
$M = M_{el} \directsum M_{le} \directsum M_{ll}$
into \'etale-local, local-\'etale, and local-local subgroup schemes.
Then 
$M^\vee = (M_{le})^\vee \directsum (M_{el})^\vee \directsum (M_{ll})^\vee$.
The homomorphism $M \to M^\vee$ induced by the alternating pairing
must map $M_{el}$ to $(M_{le})^\vee$, 
and $M_{le}$ to $(M_{el})^\vee$, 
and $M_{ll}$ to $(M_{ll})^\vee$.
In particular, $b=b_e+b_{ll}$ where $b_e$ and $b_{ll}$
are alternating pairings on $M_{el} \directsum M_{le}$ and $M_{ll}$,
respectively.
The pairing $b_e$ necessarily has the form
\[
	(m_{el},m_{le}),(m_{el}',m_{le}') 
	\mapsto B(m_{el},m_{le}') B(m_{el}',m_{le})
\]
for some bilinear pairing $B \colon M_{el} \times M_{le} \to \boldmu_2$.
Then $b_e$ comes from 
the quadratic map $(m_{el},m_{le}) \mapsto B(m_{el},m_{le})$.

It remains to consider the case $M=M_{ll}$.
Then $M^\vee(k_s)=0$.
By \cite{MilneADT2006}*{Proposition~III.6.1 and the paragraph preceding it},
$\HH^1(M^\vee)=\HH^1(G_k,M^\vee(k_s))=\HH^1(G_k,0)=0$,
so $\delta=0$.
\item Follows from~\eqref{I:factors through cyclic}.
\item The assumptions imply $k(A[2])$ is unramified over $k$
      (see~\cite{Serre-Tate1968}*{Theorem~1}, for example),
      so \eqref{I:factors through cyclic} applies.
\item Follows from \eqref{I:factors through cyclic} 
      and~\eqref{I:cT for perfect k of char 2}.
\item 
By definition of $e_2^\lambda$, 
the right kernel of $e_2^\lambda$ contains 
the kernel $K$ of $A[2] \stackrel{\lambda}\surjects \lambda(A[2])$.
Since $e_2^\lambda$ is alternating, 
the left kernel of $e_2^\lambda$ contains $K$ too.
Thus $e_2^\lambda$ induces a nondegenerate alternating pairing
\[
	e' \colon \lambda(A[2]) \times \lambda(A[2]) \to \G_m.
\]
In particular, the \'etale group scheme $\lambda(A[2])$
has square order, which by assumption is $1$ or $4$.
Let
\[
	q' \colon \lambda(A[2]) \to \G_m
\]
be the morphism taking $0$ to $1$ and all other $k_s$-points 
of $\lambda(A[2])$ to $-1$.
Then $q'$ is a quadratic form satisfying 
the identity $q'(x+y)-q'(x)-q'(y)=e'(x,y)$.
Now $q:=q' \circ \lambda$ is a quadratic form on $A[2]$
refining $e_2^\lambda$,
so $c_\lambda=0$.\qedhere
\end{enumerate}
\end{proof}

\subsection{Formula for the obstruction in the case of a line sheaf on a torsor}

Let $P$ be a torsor under $A$.
For $a \in A(S)$, let $\tau_a \colon P \to P$ be the translation.
Also, for $x \in P(S)$, let $\tau_x \colon A \to P$ be the torsor action.
The maps $\tau_x^*$ for local choices of sections $x$ 
induce a well-defined isomorphism $\PIC^0_{P/S} \isom \PIC^0_{A/S}$
since any $\tau_a^*$ is the identity on $\PIC^0_{A/S}$.
Let $\LL \in \PIC_{P/S}(S)$.
Generalizing \eqref{E:definition of phi_L}, we define
\begin{align*}
	\phi_\LL \colon A &\to \PIC^0_{P/S} \isom \PIC^0_{A/S} \\
	a &\mapsto \tau_a^* \LL \tensor \LL^{-1}.
\end{align*}
We may view $\phi_\LL$ as an element of $\Hom_{\textup{self-dual}}(A,\Adual)$.
If $x \in P(S)$, then $\phi_{\tau_x^* \LL} = \phi_\LL$.

\begin{proposition}
\label{P:formula for c_lambda}
Let $P$ be a torsor under $A$, equipped with an order-$2$ automorphism
$\iota \colon P \to P$ compatible with $[-1] \colon A \to A$.
The fixed locus $P^\iota$ of $\iota$ is a torsor under $A[2]$;
let $c \in \HH^1(A[2])$ be its class.
Let $\LL \in \PIC_{P/S}(S)$ be such that $\iota^* \LL \isom \LL$, 
and let $\lambda=\phi_\LL \colon A \to \Adual$.
Then $c_\lambda = \lambda(c)$ in $\HH^1(\Adual[2])$.
\end{proposition}

\begin{proof}
If $x$ is a section of $P^\iota$,
then $[-1]^* \tau_x^* \LL \isom \tau_x^* \iota^* \LL \isom \tau_x^* \LL$,
so we obtain a map
\begin{align*}
  \gamma \colon P^\iota &\to \PIC^{\Sym}_{A/S} \\
  x &\to \tau_x^* \LL.
\end{align*}
For sections $a \in A[2]$ and $x \in P^\iota$, we have
\[
	\gamma(a+x) = \tau_{a+x}^*\LL = \tau_a^* (\tau_x^* \LL) = \phi_{\tau_x^* \LL}(a) \tensor \tau_x^* \LL = \lambda(a) \tensor \gamma(x)
\]
in $\PIC^{\Sym}_{A/S}$.
In other words, $\gamma$ is a torsor map 
(with respect to $\lambda \colon A[2] \to \Adual[2]$)
from the torsor $P^\iota$ (under $A[2]$)
to the torsor (under $\Adual[2]$) of line sheaves
in $\PIC^{\Sym}_{A/S}$ with N\'eron-Severi class $\lambda$.
Taking classes of these torsors yields $\lambda(c)=c_\lambda$.
\end{proof}

\subsection{Application to Jacobians}
\label{S:Jacobians}

Let $X \to S$ be a family of genus-$g$ curves,
by which we mean a smooth proper morphism whose geometric fibers
are integral curves of genus $g$.
(If $g \ne 1$, then the relative canonical sheaf or its inverse
makes $X \to S$ projective:
see Remark~2 on page~252 of \cite{Bosch-Lutkebohmert-Raynaud1990}.)
By the statement and proof of Proposition~4 on page 260
of \cite{Bosch-Lutkebohmert-Raynaud1990},
\begin{enumerate}
\item There is a decomposition of functors $\PIC_{X/S} \isom \coprod_{n \in \Z} \PIC^n_{X/S}$.
\item The subfunctor $\PIC^0_{X/S}$ is (represented by) a projective abelian scheme $A$ over $S$.
\item The subfunctor $\PIC^{g-1}_{X/S}$ is (represented by) a smooth projective scheme $P$ over $S$, a torsor under $A$.  (If $g=1$, then $P=A$.)
\item The scheme-theoretic image of the ``summing'' map $X^{g-1} \to P$
is an effective relative Cartier divisor on $P$ (take this to be empty if $g=0$).  Let $\Theta$ be the associated line sheaf on $P$.
\item The homomorphism $\lambda:=\phi_\Theta \colon A \to \Adual$ is an isomorphism.
\item Define $\iota \colon P \to P$ by $\FF \mapsto \omega_{X/S} \tensor \FF^{-1}$; then $\iota^* \Theta \isom \Theta$.  (To prove this, one can reduce to the case where $S$ is a moduli scheme of curves with level structure, and then to the case where $S$ is the spectrum of a field, in which case it is a consequence of the Riemann-Roch theorem.)
\end{enumerate}

\begin{definition}
The \defi{theta characteristic torsor} $\calT$ 
is the closed subscheme of $P=\PIC^{g-1}_{X/S}$ parametrizing classes
whose square is the canonical class $\omega_{X/S} \in \PIC^{2g-2}_{X/S}(S)$.
\end{definition}

Equivalently, $\calT = P^\iota$.
Let $c_\calT \in \HH^1(A[2])$ be the class of this torsor.

\begin{theorem}
\label{T:cohomology pairing on Jacobian}
Let $X \to S$ be a family of genus-$g$ curves,
and let $A,\lambda,c_\calT$ be as above.
Then $c_\lambda=\lambda(c_\calT)$ in $\HH^1(\Adual[2])$,
and for any $x \in \HH^1(A[2])$ we have
\begin{equation}
\label{E:cup products for Jacobian}
	x \underset{e_2^\lambda}{\cup} x = x \underset{e_2^\lambda}{\cup} c_\calT
\end{equation}
in $\HH^2(\G_m)$.
\end{theorem}

\begin{proof}
Proposition~\ref{P:formula for c_lambda} 
with $P = \PIC^{g-1}_{X/S}$ and $\LL = \Theta$
yields $c_\lambda=\lambda(c_\calT)$.
So \eqref{E:cup products for abelian scheme}
in Theorem~\ref{T:cohomology pairing on abelian scheme}
becomes \eqref{E:cup products for Jacobian}.
\end{proof}

\begin{remark}
\label{R:Atiyah and Mumford}
If $S=\Spec k$ for a field $k$ of characteristic not $2$,
and the action of $G_k$ on $A[2](k_s)$ factors through a cyclic quotient,
then Proposition~\ref{P:c_lambda=0}\eqref{I:factors through cyclic} 
gives $c_\lambda=0$, so $c_\calT=0$, 
recovering the result of M.~Atiyah~\cite{Atiyah1971}*{\S5}
that under these hypotheses $X$ has a rational theta characteristic.

Similarly,
if $S=\Spec k$ for a perfect field $k$ of characteristic~$2$,
then Proposition~\ref{P:c_lambda=0}\eqref{I:cT for perfect k of char 2}
gives $c_\lambda=0$, so $c_\calT=0$.
In fact, the proof produces a canonical $k$-point of $\calT$.
This generalizes an observation of Mumford~\cite{Mumford1971}*{p.~191} 
that a curve over an algebraically closed field of characteristic~$2$
has a canonical theta characteristic.

Additional criteria for the existence of a rational theta characteristic
are given in~\cite{Sharif2011-preprint}.
\end{remark}

\subsection{Hyperelliptic Jacobians}
\label{S:hyperelliptic Jacobians}

\begin{proposition}
\label{P:cT for hyperelliptic curve}
If $E$ is an elliptic curve, then $x \underset{e_2^\lambda}{\cup} x = 0$
for all $x \in \HH^1(E[2])$.
The same holds for the Jacobian of any hyperelliptic curve $X$
if it has a rational Weierstrass point or its genus is odd.
In particular, this applies to $y^2=f(x)$ with $f$ separable of
degree $n \not\equiv 2 \pmod{4}$ over a field
of characteristic not $2$.
\end{proposition}

\begin{proof}
For an elliptic curve $E$, the trivial line sheaf $\OO_E$ 
is a theta characteristic defined over $k$.
Now suppose that $X$ is a hyperelliptic curve of genus $g$,
so it is a degree-$2$ cover of a genus-$0$ curve $Y$.
The class of a point in $Y(k_s)$
pulls back to a $k$-point of $\PIC_{X/k}^2$,
and if $g$ is odd, multiplying by $(g-1)/2$ gives a $k$-point of $\calT$.
On the other hand, if $X$ has a rational Weierstrass point $P$,
then $\OO((g-1)P)$ is a theta characteristic defined over $k$.
So $\calT$ is trivial in all these cases.
Now apply Theorem~\ref{T:cohomology pairing on Jacobian}.
\end{proof}

\begin{example} \label{E:genus 2}
Suppose that $X$ is a genus~$2$ curve,
the smooth projective model of $y^2=f(x)$ 
where $f$ is a degree-$6$ separable polynomial
over a field $k$ of characteristic not $2$.
Let $\Delta$ be the set of zeros of $f$ in $k_s$.
As explained in~\cite{Mumford1971}*{p.~191},
the group $A[2]$ and its torsor $\calT$
can be understood explicitly in terms of $\Delta$.
Namely, for $m \in \Z/2\Z$, 
let $\calW_m$ be the quotient of the sum-$m$ part of
the permutation module $\F_2^\Delta\isom \F_2^{2g+2}$ 
by the diagonal addition action of $\F_2$.
Then the $G_k$-module $A[2]$ may be identified with $\calW_0$,
and its torsor $\calT$ may be identified with $\calW_1$.
Using this, one can show:
\begin{enumerate}[\upshape (a)]
\item For $f(x)=(x^2+1)(x^2-3)(x^2+3)$ over $\Q_3$, we have $c_\calT \ne 0$.
\item\label{I:nonzero Sha}
 For $f(x)=x^6+x+6$ over $\Q$, 
we have 
\[
0 \ne c_\calT \in \Sha^1(\Q,A[2])
:=\ker\left( \HH^1(\Q,A[2]) \to \prod_{p \le \infty} \HH^1(\Q_p,A[2]) \right).
\]
({\em Proof:}
The discriminant of $f$ is $-\ell$, where $\ell$ is the prime $362793931$.
For $p \notin \{2,\ell\}$,
the element $c_\calT$ maps to $0$ in $\HH^1(\Q_p,A[2])$
by Proposition~\ref{P:c_lambda=0}(\ref{I:cT for archimedean k},\ref{I:cT for local k}),
and $f(x)$ has a zero in each of $\Q_2$ and $\Q_\ell$,
so the same is true at those places.
On the other hand, the Galois group of $f$ over $\Q$ is $S_6$,
so $c_\calT \ne 0$.)
\end{enumerate}
\end{example}

%****************************************************************************
\section*{Acknowledgements} 

We thank Benedict Gross for the suggestion to look at 
the self cup product on $\HH^1(k,A[2])$ induced by the Weil pairing $e_2$
on a Jacobian.
We also thank Brian Conrad for comments.


\begin{bibdiv}
\begin{biblist}
% \bibselect{big}

\bib{Atiyah1971}{article}{
  author={Atiyah, Michael F.},
  title={Riemann surfaces and spin structures},
  journal={Ann. Sci. \'Ecole Norm. Sup. (4)},
  volume={4},
  date={1971},
  pages={47--62},
  issn={0012-9593},
  review={\MR {0286136 (44 \#3350)}},
}

\bib{Bosch-Lutkebohmert-Raynaud1990}{book}{
  author={Bosch, Siegfried},
  author={L{\"u}tkebohmert, Werner},
  author={Raynaud, Michel},
  title={N\'eron models},
  series={Ergebnisse der Mathematik und ihrer Grenzgebiete (3) [Results in Mathematics and Related Areas (3)]},
  volume={21},
  publisher={Springer-Verlag},
  place={Berlin},
  date={1990},
  pages={x+325},
  isbn={3-540-50587-3},
  review={\MR {1045822 (91i:14034)}},
}

\bib{Faltings-Chai1990}{book}{
  author={Faltings, Gerd},
  author={Chai, Ching-Li},
  title={Degeneration of abelian varieties},
  series={Ergebnisse der Mathematik und ihrer Grenzgebiete (3) [Results in Mathematics and Related Areas (3)]},
  volume={22},
  note={With an appendix by David Mumford},
  publisher={Springer-Verlag},
  place={Berlin},
  date={1990},
  pages={xii+316},
  isbn={3-540-52015-5},
  review={\MR {1083353 (92d:14036)}},
}

\bib{Giraud1971}{book}{
  author={Giraud, Jean},
  title={Cohomologie non ab\'elienne},
  language={French},
  note={Die Grundlehren der mathematischen Wissenschaften, Band 179},
  publisher={Springer-Verlag},
  place={Berlin},
  date={1971},
  pages={ix+467},
  review={\MR {0344253 (49 \#8992)}},
}

\bib{Hatcher-Algebraic-Topology}{book}{
  author={Hatcher, Allen},
  title={Algebraic topology},
  date={2010-06-30},
  note={Downloaded from \url {http://www.math.cornell.edu/~hatcher/AT/ATpage.html}\phantom {m}},
}

\bib{MilneADT2006}{book}{
  author={Milne, J. S.},
  title={Arithmetic duality theorems},
  publisher={BookSurge, LLC},
  edition={Second edition},
  date={2006},
  pages={viii+339},
  isbn={1-4196-4274-X},
  review={\MR {881804 (88e:14028)}},
}

\bib{MumfordAV1970}{book}{
  author={Mumford, David},
  title={Abelian varieties},
  series={Tata Institute of Fundamental Research Studies in Mathematics, No. 5 },
  publisher={Published for the Tata Institute of Fundamental Research, Bombay},
  date={1970},
  pages={viii+242},
  review={\MR {0282985 (44 \#219)}},
}

\bib{Mumford1971}{article}{
  author={Mumford, David},
  title={Theta characteristics of an algebraic curve},
  journal={Ann. Sci. \'Ecole Norm. Sup. (4)},
  volume={4},
  date={1971},
  pages={181--192},
  issn={0012-9593},
  review={\MR {0292836 (45 \#1918)}},
}

\bib{Polishchuk2003}{book}{
  author={Polishchuk, Alexander},
  title={Abelian varieties, theta functions and the Fourier transform},
  series={Cambridge Tracts in Mathematics},
  volume={153},
  publisher={Cambridge University Press},
  place={Cambridge},
  date={2003},
  pages={xvi+292},
  isbn={0-521-80804-9},
  review={\MR {1987784 (2004m:14094)}},
}

\bib{Poonen-Rains-selmer-preprint}{misc}{
  author={Poonen, Bjorn},
  author={Rains, Eric},
  title={Random maximal isotropic subspaces and Selmer groups},
  date={2011-04-10},
  note={Preprint},
}

\bib{Poonen-Stoll1999}{article}{
  author={Poonen, Bjorn},
  author={Stoll, Michael},
  title={The Cassels-Tate pairing on polarized abelian varieties},
  journal={Ann. of Math. (2)},
  volume={150},
  date={1999},
  number={3},
  pages={1109\ndash 1149},
  issn={0003-486X},
  review={\MR {1740984 (2000m:11048)}},
}

\bib{Serre-Tate1968}{article}{
  author={Serre, Jean-Pierre},
  author={Tate, John},
  title={Good reduction of abelian varieties},
  journal={Ann. of Math. (2)},
  volume={88},
  date={1968},
  pages={492--517},
  issn={0003-486X},
  review={\MR {0236190 (38 \#4488)}},
}

\bib{SGA6}{book}{
  title={Th\'eorie des intersections et th\'eor\`eme de Riemann-Roch},
  language={French},
  series={Lecture Notes in Mathematics, Vol. 225},
  note={S\'eminaire de G\'eom\'etrie Alg\'ebrique du Bois-Marie 1966--1967 (SGA 6); Dirig\'e par P. Berthelot, A. Grothendieck et L. Illusie. Avec la collaboration de D. Ferrand, J. P. Jouanolou, O. Jussila, S. Kleiman, M. Raynaud et J. P. Serre},
  publisher={Springer-Verlag},
  place={Berlin},
  date={1971},
  pages={xii+700},
  review={\MR {0354655 (50 \#7133)}},
  label={SGA 6},
}

\bib{Sharif2011-preprint}{misc}{
  author={Sharif, Shahed},
  title={A descent map for curves with totally degenerate semi-stable reduction},
  date={2011-03-14},
  note={Preprint},
}

\bib{Waterhouse1971}{article}{
  author={Waterhouse, William C.},
  title={Principal homogeneous spaces and group scheme extensions},
  journal={Trans. Amer. Math. Soc.},
  volume={153},
  date={1971},
  pages={181--189},
  issn={0002-9947},
  review={\MR {0269659 (42 \#4554)}},
}

\bib{Yoneda1958}{article}{
  author={Yoneda, Nobuo},
  title={Note on products in ${\rm Ext}$},
  journal={Proc. Amer. Math. Soc.},
  volume={9},
  date={1958},
  pages={873--875},
  issn={0002-9939},
  review={\MR {0175957 (31 \#233)}},
}

\bib{Zarhin1974}{article}{
  author={Zarhin, Ju. G.},
  title={Noncommutative cohomology and Mumford groups},
  language={Russian},
  journal={Mat. Zametki},
  volume={15},
  date={1974},
  pages={415--419},
  issn={0025-567X},
  review={\MR {0354612 (50 \#7090)}},
}

\end{biblist}
\end{bibdiv}
\end{document}